\documentclass[11pt,twoside,letterpaper]{article}%
\usepackage{bm}
\usepackage{makeidx}
\usepackage{amssymb}
\usepackage{amsfonts}
\usepackage{graphicx}
\usepackage{amsmath}
\usepackage{fancyhdr}
\usepackage{latexsym} 
\usepackage{geometry}
\usepackage{setspace}
\usepackage{color}%

\providecommand{\U}[1]{\protect\rule{.1in}{.1in}}
\newtheorem{theorem}{Theorem}
\newtheorem{corollary}{Corollary}

\newtheorem{lemma}{Lemma}
\newtheorem{proposition}{Proposition}

\def \C{\mathbb{C}}
\def \H{\mathbb{H}}
\def \N{\mathbb{N}}

\def \Q{\mathbb{Q}}

\def \K{\mathbb{K}}
\def \Z{\mathbb{Z}}

\begin{document}

\title{On Algebraic Conditions for the Non-Vanishing of Linear Forms in Jacobi Theta-Constants}
\author{
\textsc{
Carsten Elsner\thanks{Institute of Computer Sciences, FHDW University of Applied Sciences, Freundallee 15, 30173 Hannover, Germany 
\newline e-mail: carsten.elsner@fhdw.de  (Corresponding author)},\,\,\, 
Veekesh Kumar\thanks{Department of Mathematics, Indian Institute of Technology, Dharwad  580007, Karnataka, India.
\newline e-mail: veekeshk@iitdh.ac.in}
}}
\date{}
\maketitle

\begin{abstract}
Elsner, Luca and Tachiya proved in 2019 that  the  values of the Jacobi-theta constants $\theta_3(m\tau)$ and $\theta_3(n\tau)$ are algebraically independent over
${\Q}$ for distinct integers $m,n$ under some conditions on $\tau$. On the other hand, in 2018 Elsner and Tachiya also proved that three values
$\theta_3(m\tau),\theta_3(n\tau)$ and $\theta_3(\ell \tau)$ are algebraically dependent over ${\Q}$. In this article we prove the non-vanishing of linear forms in
$\theta_3(m\tau)$, $\theta_3(n\tau)$ and $\theta_3(\ell \tau)$ under various conditions on $m,n,\ell$, and $\tau$. Among other things we prove that for odd and distinct
positive integers $m,n>3$ the three numbers $\theta_3(\tau)$, $\theta_3(m\tau)$ and $\theta_3(n \tau)$ are linearly independent over $\overline{\mathbb{Q}}$ when
$\tau$ is an algebraic number of some degree greater or equal to 3. In some sense this fills the gap between the above-mentioned former results on theta constants. A
theorem on the linear independence over $\mathbb{C(\tau)}$ of the functions $\theta_3(a_1 \tau),\ldots,\theta_3(a_m \tau)$ for distinct positive rational numbers 
$a_1, \ldots a_m$ is also established.
\end{abstract}

\thispagestyle{fancy} \fancyhead{} \fancyhead[L]{In: Book Title \\
Editor: Editor Name, pp. {\thepage-\pageref{lastpage-01}}}
\fancyhead[R]{ISBN 0000000000  \\
\copyright~2007 Nova Science Publishers, Inc.} \fancyfoot{} \renewcommand{\headrulewidth}{0pt}

\thispagestyle{empty} \setcounter{page}{1}

\hspace*{59pt} modular functions.


\pagestyle{fancy} \fancyhead{}
\fancyhead[EC]{Carsten Elsner and Vekesh Kumar} \fancyhead[EL,OR]{\thepage}
\fancyhead[OC]{On Algebraic Conditions for the Non-Vanishing of Linear Forms in Jacobi Theta-Constants}
\fancyfoot{} \renewcommand\headrulewidth{0.5pt}

\section{Introduction}
\label{Sec1}
For a complex number $\tau$ from the upper complex half plane $\mathbb{H}$, the theta functions are defined as follows;
$$
\theta_2(\tau)=2\sum_{n=0}^{\infty} q^{{(n+1/2)}^2} \,,\qquad \theta_3(\tau)=1+2\sum_{n=1}^\infty q^{n^2}\,,  \quad\mbox{~~and~~}
\quad\theta_4(\tau)=1+2\sum_{n=1}^\infty (-1)^nq^{n^2}\,, 
$$
where $q=e^{i\pi\tau}$. For the sake of brevity we sometimes write $\theta_i$ instead of $\theta_i(\tau)$, $i=2,3,4$.
\smallskip

We define the $j$- function as follows;
$$
j(\tau)=256\frac{(\lambda^2-\lambda+1)^3}{\lambda^2(\lambda-1)^2},\quad \mbox{~~where~~} \lambda=\lambda(\tau)=\frac{\theta^4_2}{\theta^4_3},
$$
which is a modular function with respect to the group $SL(2,\mathbb{Z})$.
\bigskip

The motivation of this article comes from the following sources: In $2018$, C.Elsner and Y.Tachiya \cite{elsner} proved that for distinct integers $\ell, m$ and
$n$, the functions $\theta_3(\ell\tau), \theta_3(m\tau)$ and $\theta_3(n\tau)$ are {\em algebraically\/} dependent over $\mathbb{Q}$.  In 2019, 
C.\,Elsner, F.\,Luca and Y.\,Tachiya \cite{luca} proved the following: 
let $\tau$ be any complex number with Im$(\tau) >0$ such that $e^{i\pi\tau}$ is algebraic. Let $m,n\geq 1$ be distinct positive integers. Then the numbers 
$\theta_3(m\tau)$ and $\theta_3(n\tau)$ are {\em algebraically} independent over $\mathbb{Q}$.  In a subsequent paper \cite{elsner2} from 2020 this
occurs as a special case of \cite[Theorem 1.1]{elsner2}. 
\smallskip

Naturally the following two questions arise.
\smallskip

\noindent{\bf Question 1.~}
Let $m\geq 2$ and let $a_1, a_2, \ldots, a_m$ be distinct positive integers. Are the functions 
$$
\theta_3(a_1 \tau),\,\theta_3(a_2\tau),\,\ldots,\,\theta_3(a_m\tau)
$$
{\em linearly\/} independent over  $\mathbb{C(\tau)}$? \\
\smallskip
By \cite[Theorem\,1.1]{elsner2} we know that for distinct  positive integers $m,n$, and an algebraic number $e^{i\pi \tau}$ with $\tau \in \mathbb{H}$, the two 
numbers $\theta_3(m\tau)$ and $\theta_3(n\tau)$ are {\em algebraically\/} independent over ${\Q}$. Then, for algebraic numbers $\alpha_1$ and
$\alpha_2$, which do not vanish simultaneously, the linear form 
$$
\alpha_1\theta_3(m\tau) + \alpha_2\theta_3(n\tau)
$$
does not vanish. So it seems natural to consider linear forms 
in three values of the theta constant $\theta_3$.
\bigskip
 
\noindent{\bf Question 2.~}

What are the values of $\tau$ and $\alpha_0,\alpha_1,\alpha_2$ such that the linear form  
$$
L \,:=\, \alpha_0\theta_3(\tau) + \alpha_1\theta_3(m\tau) + \alpha_2\theta_3(n\tau)
$$ 
does not vanish?
\bigskip

In this article, we give the complete answer to the Question\,1 and answer the Question\,2 in the following way: We consider linear forms with
integers $n>m>1$, certain numbers $\tau \in \mathbb{H}$, and algebraic numbers $\alpha_0,\alpha_1, \alpha_2$, and give conditions on 
$\alpha_0,\alpha_1, \alpha_2$ such that $L\not= 0$. \\
We divide the remaining part of our article into three sections:
In Section\,2,  we  state our theorems, in Section\,3 we collect all the tools to prove our results, and in the last section we give the proof of all the theorems
from Section\,2.


\section{The results}
\label{Sec2}
\subsection{The linear independence over ${\C}(\tau)$ of the functions $\theta_3(a_1 \tau),\,\ldots,\,\theta_3(a_m\tau)$ in $\tau$} 
\[\]
We begin with the following result on the linear independence over ${\C}(\tau)$ of Jacobi-theta constants.
\begin{theorem}\label{maintheorem}
{\em (i)\/}\,\,Let $a_1, a_2, \ldots, a_m$ be distinct positive rational numbers. Then the $m$ functions 
$$
\theta_3(a_1 \tau), \quad \theta_3(a_2\tau),\ldots, \theta_3(a_m\tau)
$$ 
in $\tau \in {\H}$ are linearly independent over $\mathbb{C(\tau)}$. \\
{\em (ii)}\/\,\,Let $\alpha_1, \alpha_2, \ldots, \alpha_m$ be distinct positive real numbers. Then the $m$ functions 
$$
\theta_3(\alpha_1 \tau), \quad \theta_3(\alpha_2\tau),\ldots, \theta_3(\alpha_m\tau)
$$ 
are linearly independent over ${\C}$.
\end{theorem}

\subsection{On the linear independence of values of Jacobi-theta constants for \\ $\theta_3(\tau) ,\,\theta_3(m\tau),\,\theta_3(n\tau)$ with odd integers $m,n$}
\begin{theorem}
Let $3\leq n<m$ be two odd integers. If one of the following conditions holds, namely
\begin{enumerate}
\item 
$\tau\in\mathbb{H}$ is an algebraic number of degree $\geq 3$,
\item
$\tau\in\mathbb{H}$ such that $q=e^{i\pi\tau}$ is algebraic over $\mathbb{Q}$\,,
\end{enumerate}
then the three numbers 
$$ 
\theta_3(\tau),\quad\theta_3(m\tau),\quad\theta_3(n\tau)
$$
are  linearly independent over $\overline{\mathbb{Q}}$.
\label{Thm3.3.3}
\end{theorem}

\subsection{Results on linear forms $\alpha_0\theta_3(\tau) + \alpha_1\theta_3(m\tau) + \alpha_2\theta_3(n\tau)$ for $mn\equiv 0 \pmod 2$ under certain restrictions 
on the coefficients}
\[\]
Let $m,n$ be two distinct  positive integers, and let $\tau\in\mathbb{H}$ satisfying the conditions in Theorem\,\ref{Thm3.3.3}. Since the numbers $\theta_3(m\tau)$ 
and $\theta_3(n\tau)$ are algebraically independent when $e^{i\pi \tau}$ is algebraic by Theorem\,1.1 in \cite{elsner2}, we consider linear relations
\[\alpha_0 \theta_3(\tau) + \alpha_1\theta_3(m\tau) + \alpha_2\theta_3(n\tau) \]
for real algebraic numbers $\alpha_1,\alpha_2,\alpha_3$ with $\alpha_0\alpha_1\alpha_2 \not=0$. 
In order to state our next result we introduce the following set. Let $s\geq 3$ be any odd integer. Set
\[M_s \,:=\, \big\{ \pm \sqrt{u}, \pm i \sqrt{u}\,:\,\, u \in {\N}\,\wedge \, s\,\equiv \, 0 \pmod u \big\} \,.\]
\begin{theorem}
Let $m=2^a s_1$ and $n=2^b s_2$ be two distinct integers with $a,b\geq 1$ and odd integers $s_1,s_2\geq 3$. Let $\tau\in\mathbb{H}$ such that $e^{i\pi \tau}$
is an algebraic number. Then, the inequality
$$ 
\alpha_0\theta_3(\tau) + \alpha_1\theta_3(m\tau) + \alpha_2\theta_3(n\tau) \,\not= \, 0
$$ 
holds, if $\beta =\alpha_2\alpha_0^{-1}$ satisfies at least one of the following conditions:
$$
\beta^{-1} \,\not\in \, M_{s_2},\quad
\beta^4 \,\not\in \, {\Q}, \quad
 R_{n,0}(\beta^{-4})\neq 0 \,,
$$
where $R_{n,0}(X)$ is the polynomial defined by formula (\ref{eq2.1}) in Lemma\,\ref{Lem3.3.1} in Section 3.
\label{Thm3.3.4}
\end{theorem}
In the case when additionally $s_1$ and $s_2$ are coprime odd integers, we have $M_{s_1} \cap M_{s_2} = \{ \pm 1,\,\pm i \}$. Then we obtain from 
Theorem\,\ref{Thm3.3.4} the following corollary.
\begin{corollary}
Let $m=2^a s_1$ and $n=2^b s_2$ be two distinct  integers with $a,b\geq 1$ and odd coprime integers $s_1,s_2\geq 3$. Let $\tau\in\mathbb{H}$ such 
that $e^{i\pi \tau}$ is an algebraic number. Then, the inequality
$$
\alpha_0\theta_3(\tau) + \alpha_1\theta_3(m\tau) + \alpha_2\theta_3(n\tau) \,\not= \, 0
$$
holds for all algebraic numbers $\alpha_0,\alpha_1,\alpha_2$, whenever at least one of the numbers $\alpha_0/\alpha_1$ or $\alpha_0/\alpha_2$ is not a  
unit in  the ring of  Gaussian integers ${\Z}[i]$.
\label{Cor2.1}
\end{corollary}
For any  integer $n>1$ let
$$
\psi(n):= n\prod_{p|n}\left(1+\frac{1}{p}\right)\,,
$$ 
where $p$ runs through all primes dividing $n$ and for $n=1$, we define $\psi(1)=1$.
\begin{theorem}
Let $m=2^a s$ be an integer   with $a\geq 1$ and an odd integer $s\geq 3$ and let $n\geq 3$ be an odd integer. Let $\tau\in\mathbb{H}$ be as in Theorem\,\ref{Thm3.3.3}.
Then, the inequality
$$
\alpha_0\theta_3(\tau) + \alpha_1\theta_3(m\tau) + \alpha_2\theta_3(n\tau) \,\not= \, 0
$$
holds, if $\beta := \alpha_2\alpha_0^{-1}$ satisfies one of the following conditions:
$$
 \deg_{\Q} \big( \beta^4\,\big) \,>\, \psi(n),\quad
S_{n,0}(n^2\beta^{-4})S_{n,d_n}(n^2\beta^{-4}) \,\not= \, 0\,,
$$
where $S_{n,0}(X)$ and $S_{n,d_n}(X)$ are polynomials defined in (\ref{EQN1}) in Theorem\,\ref{Thm3.3} below.
\label{Thm3.3.5}
\end{theorem}
The preceding theorems do not treat $\theta_3(\tau),\theta_3(2\tau)$, and $\theta_3(3\tau)$ simultaneously. For this situation we cite a result from 
\cite[Example\,1.5]{elsner}: Let $\tau \in \mathbb{H}$, and define
\[P(X,Y,Z) \,:=\, 27X^8 - 18X^4Y^4 - 64X^2Y^4Z^2 + 64X^2Y^2Z^4 - 8X^2Z^6 - Z^8 \,.\]
Then $P(X_0,Y_0,Z_0)=0$ holds for
\[X_0 \,:=\, \theta_3(3\tau)\,,\quad Y_0 \,:=\, \theta_3(2\tau)\,,\quad Z_0 \,:=\, \theta_3(\tau)\,.\]
This shows that $\theta_3(\tau),\theta_3(2\tau),\theta_3(3\tau)$ are homogeneously algebraically dependent of degree 8. 
\begin{proposition}
Let $m\geq0$ be an integer. Let $\tau\in\mathbb{H}$ be as in Theorem\,\ref{Thm3.3.3}. Then the three numbers 
$$
\theta_3(2^m\tau),\quad \theta_3(2^{m+1}\tau),\quad \theta_3(2^{m+2}\tau)
$$ 
are  linearly independent over $\overline{\mathbb{Q}}$.
\label{Prop1}
\end{proposition}
In Proposition\,\ref{Prop1} it is not possible to avoid the condition $\deg_{\Q}(\tau)\geq 3$. This follows for $\tau=i$ and $m=0$ from the nontrivial relation
\[\sqrt{2+\sqrt{2}}\,\theta_3(i) - 2\theta_3(2i) \,=\, 0 \]
due to Ramanujan, cf. \cite[p.\,325]{Berndt}. We can also find similar linear relations for the theta-constants $\theta_2$ and $\theta_4$:
\[\sqrt{2-\sqrt{2}}\,\theta_2(i) - \sqrt[4]{8}\,\theta_2(2i) \,=\, 0 \,,\qquad \sqrt[8]{2}\,\theta_4(i)- \theta_4(2i) \,=\, 0\,.\]
Moreover, we have for $\tau = 1+i\sqrt{3}$ (with $q=-e^{-\pi \sqrt{3}}$\,) the following linear identity involving values of $\theta_2$ and $\theta_3$,
\[2\theta_2\big( \,1+i\sqrt{3}\,\big) - (1+i)\sqrt[4]{28-16\sqrt{3}}\,\theta_3(1+i\sqrt{3}) \,=\, 0 \,.\]


\section{Main Tools towards the proof of our results}
\label{Sec3}
\begin{theorem} \cite[page 5]{Schneider} \\ 
Let $\tau\in\mathbb{H}$ be an algebraic number of  degree different than  $2$. Then $j(\tau)$ is transcendental.
\label{Thm3.1}
\end{theorem}
\begin{theorem} \cite[Theorem 4]{Bertrand} \\
For any $\tau\in\mathbb{H}$,  if $q=e^{i\pi\tau}$ is an algebraic number, and for integers $j,k,\ell \in \{ 2,3,4\}$ with $j\not= k$, then the three values 
$\theta_j(\tau)$, $\theta_k(\tau)$  and $D\theta_{\ell}(\tau)$ are algebraically independent over $\mathbb{Q}$. Here, 
\[D \,:=\, \frac{1}{\pi i}\frac{d}{d\tau} \]
is a differential operator.
\label{Thm3.2}
\end{theorem}
\begin{theorem} \cite[Theorem 1, Corollary 4]{nest2} \\  
For any odd integer $n\geq 3$ there exists an integer polynomial $P_n(X,Y)$ with deg${_X}~ P_n(X,Y)=\psi(n)$ such that 
$$
P_n\left(n^2\frac{\theta^4_3(n\tau)}{\theta^4_3(\tau)}, 16\frac{\theta^4_2(\tau)}{\theta^4_3(\tau)}\right)=0
$$
holds for all $\tau\in\mathbb{H}$.
Moreover, the polynomial $P_n(X, Y)$ is of the form
\begin{equation*}\label{eq1.1}
\tag{3.1}
P_n(X, Y):= X^{\psi(n)}+R_1(Y)X^{\psi(n)-1}+\cdots+R_{\psi(n)-1}(Y)X+R_{\psi(n)}(Y) \,=\, \sum_{k=0}^{\psi(m)} R_k(Y)X^{\psi(m)-k},
\end{equation*}
where $R_j(Y)\in\mathbb{Z}[Y]$ for $j=1,\dots,\psi(n)$, and
\[\deg R_k(Y) \,\leq \, k\cdot \frac{n-1}{n} \qquad \big( 1\leq k\leq \psi(n)\big) \,.\]
Let 
\[d_n \,:=\, \max_{1\leq k \leq \psi(n)} \deg_Y R_k(Y) \,.\] 
Then, $P_n(X,Y)$ can be written as
\begin{equation}
\tag{3.2}
P_n(X,Y) \,=\, \sum_{j=0}^{d_n} S_{n,j}(X)Y^j \,,
\label{EQN1}
\end{equation}
where $S_{n,j}(X) \in {\Z}[X]$ with $0\leq j\leq d_n$ and $d_n>0$, such that
\begin{equation}
\tag{3.3}
S_{n,j}(0) \,=\, 0 \quad (1\leq j\leq d_n)\,,\qquad S_{n,d_n}(X) \,\not\equiv \, 0 \,,
\label{EQN2}
\end{equation}
\begin{equation}
\tag{3.4}
S_{n,0}(0) \,=\, P_n(0,0) \,\in \ {\Z}\setminus \{ 0\} \,.
\label{EQN3}
\end{equation}
\label{Thm3.3}
\end{theorem} 
\smallskip
The properties (\ref{EQN1}) to (\ref{EQN3}) of $P_n(X,Y)$ follow from the proof of \cite[Lemma 7]{luca}; cf. formula (9).

\noindent{\bf Remark 1.}
From \eqref{eq1.1}, we can observe that for any complex number $\alpha$, the polynomial $P_n(X, \alpha)$  is non-zero.
\bigskip

The following lemmas are  crucial for the proof of our results.
\begin{lemma} \cite[Lemma 2.5]{elsner} \\
Let $n=2^a s$ be an integer with $a\geq 1$ and an odd integer $s\geq 3$.  Then there exists a polynomial $Q_n(X, Y)$ with integral coefficients 
such that $$
Q_n\left(\frac{\theta^4_3(n\tau)}{\theta^4_3(\tau)}, \frac{\theta^4_2(\tau)}{\theta^4_3(\tau)}\right)=0
$$
holds for all $\tau\in\mathbb{H}$.  Moreover, the polynomial $Q_n(X, Y)$ is of the form
\begin{equation*}\label{eq2.1}
\tag{3.5}
Q_n(X, Y)=c_n^{2^a}Y^{2^a \psi(s)}+\sum_{j=0}^{2^a \psi(s)-1}R_{n,j}(X)Y^j, 
\end{equation*}
where $c_n$ is a non-zero integer. Moreover,
\begin{equation}\label{eq2.2}
\tag{3.6}
\deg R_{n,j}(X) \,\leq \, 2^a \psi(s) -j \quad \big( 0\leq j<2^a\psi(s) \big) \,,
\end{equation}
and
\begin{equation*}\label{eq2.3}
\tag{3.7}
Q_n(0, Y)=c_n^{2^a}Y^{2^a \psi(s)}\,,
\end{equation*} 
\begin{equation*}\label{eq2.4}
\tag{3.8}
R_{n,0}(X)  \,=\, Q_n(X,0) \,=\, 2^{4(2^a-1)\psi(s)}X^{(2^a-1)\psi(s)}P_s(X,0) \,.
\end{equation*}
\label{Lem3.3.1}
\end{lemma}
{\em Proof.\/} \,Apart from formula (\ref{eq2.4}) the statements are given in \cite[Lemma\,2.5]{elsner}. It remains to prove (\ref{eq2.4}). 
We proceed by induction with respect to $a$ and follow the lines of the proof of \cite[Lemma\,2.5]{elsner}. 
As in the first part of the proof of Lemma\,2.5 (corresponding to $a=\alpha =1$) we construct the polynomials $B_{2s}(X,Y)$, $\tilde{Q}_{2s}(X,Y)$,
and $Q_{2s}(X,Y)$, where
\[\tilde{Q}_{2s}(X^4,Y^4) \,=\, B_{2s}(X,Y)B_{2s}(X,iY) \]
and
\[Q_{2s}(X,Y) \,=\, \tilde{Q}_{2s}(X,1-Y) \,.\]
From the proof of Lemma\,2.5 we obtain
$$B_{2s}(X,1) = 2^{2\psi(s)}P_s(X^4,0) \,,\quad
B_{2s}(X,i) = 2^{2\psi(s)}X^{4\psi(s)} \,.
$$
We sketch the proof of these two identities. We take the polynomials $B_{2s}(X,Y)$ from formula (19) in \cite{elsner} by setting $n=2s$:
\[B_{2s}(X,Y) \,=\, \sum_{\nu,\mu} 2^{2\nu}a_{\nu,\mu}X^{4\nu}{\big(1-Y^2\big)}^{2\mu}{\big(1+Y^2\big)}^{2(\psi(s)-\nu-\mu)} \,;\]
where the variables $a_{\nu,\mu}$ are the coefficients of the polynomial
\[P_s(X,Y) \,=\, \sum_{\nu,\mu} a_{\nu,\mu}X^{\nu}Y^{\mu} \]
from Theorem\,\ref{Thm3.3}. Thus, on the one hand, we obtain
\[
B_{2s}(X,1) \,=\, \sum_{\nu} 2^{2\nu}a_{\nu,0} X^{4\nu}2^{2(\psi(s)-\nu)}  \,=\, 2^{2\psi(s)}\sum_{\nu}a_{\nu,0} X^{4\nu}  
\,=\, 2^{2\psi(s)}P_s\big( X^4,0 \big)\,, \]
on the other hand,
\begin{eqnarray*}
B_{2s}(X,i) &=& \sum_{\scriptsize{\begin{array}{c}
\nu,\mu \\ \nu + \mu = \psi(s)
\end{array}}} 2^{2\nu}a_{\nu,\mu}X^{4\nu}2^{2\mu} \\
&=& 2^{2\psi(s)}\sum_{\scriptsize{\begin{array}{c}
\nu,\mu \\ \nu + \mu = \psi(s)
\end{array}}} a_{\nu,\mu}X^{4\nu} \\
&=& 2^{2\psi(s)}X^{4\psi(s)} \,.
\end{eqnarray*}
The last identity follows from the fact that $a_{\nu,\mu}=0$ for $\nu+\mu=\psi(s)$ with $\mu \geq 1$, which is a consequence of (\ref{eq1.1}) and
$\deg R_k(Y)<k$ for $k\geq 1$.  For $\mu =0$ and $\nu=\psi(s)$ we have $a_{0,\psi(s)}=R_0(Y)=1$, cf. \cite[p.\,154]{nest2}. \vspace*{5pt} \\
This gives
\[\tilde{Q}_{2s}(X^4,1) \,=\, B_{2s}(X,1)B_{2s}(X,i) \,=\, 2^{4\psi(s)}X^{4\psi(s)}P_s(X^4,0) \,,\]
and, consequently,
\[Q_{2s}(X,0) \,=\, \tilde{Q}_{2s}(X,1) \,=\, 2^{4\psi(s)} X^{\psi(s)}P_s(X,0) \,.\]
This shows that (\ref{eq2.4}) holds for $a=1$. Next, let (\ref{eq2.4}) be already proven for some fixed $a\geq 1$. For the induction step we
construct the polynomials $B_{2^{a+1}s}(X,Y)$, $\tilde{Q}_{2^{a+1}s}(X,Y)$, and $Q_{2^{a+1}s}(X,Y)$, where
$$
\tilde{Q}_{2^{a+1}s}(X^4,Y^4) \,=\, B_{2^{a+1}s}(X,Y)B_{2^{a+1}s}(X,iY), \quad
Q_{2^{a+1}s}(X,Y) \,=\, \tilde{Q}_{2^{a+1}s}(X,1-Y) \,.
$$
Since $\deg_X Q_{2^as}(X,Y) = 2^a \psi(s)$, we obtain by applying the induction hypothesis,
\begin{eqnarray*}  
B_{2^{a+1}s}(X,1) &=& 2^{2\cdot 2^a \psi(s)} Q_{2^as}(X^4) \,=\, 2^{2^{a+1}\psi(s)} \big( \,2^{4(2^a-1)\psi(s)}X^{4(2^a-1)\psi(s)}P_s(X^4,0)\,\big) \,,\\
B_{2^{a+1}s}(X,i) &=& 2^{2\cdot 2^a \psi(s)} X^{4\cdot 2^a \psi(s)} \,=\, 2^{2^{a+1}\psi(s)}X^{4\cdot 2^a \psi(s)} \,. 
\end{eqnarray*}
Therefore, it turns out that
\begin{eqnarray*}
\tilde{Q}_{2^{a+1}s}(X^4,1) &=& B_{2^{a+1}s}(X,1)B_{2^{a+1}s}(X,i) \\
&=& 2^{2\cdot 2^{a+1}\psi(s) + 4(2^a-1)\psi(s)}X^{4\cdot 2^a\psi(s) + 4(2^a-1)\psi(s)}P_s(X^4,0) \\
&=& 2^{4(2^{a+1}-1)\psi(s)}X^{4(2^{a+1}-1)\psi(s)}P_s(X^4,0) \,.
\end{eqnarray*}
We complete the proof of the lemma by observing that
\[Q_{2^{a+1}s}(X,0) \,=\, \tilde{Q}_{2^{a+1}s}(X,1) \,=\, 2^{4(2^{a+1}-1)\psi(s)}X^{(2^{a+1}-1)\psi(s)}P_s(X,0) \,,\]
which is the identity in (\ref{eq2.4}) with $a$ replaced by $a+1$. \hfill $\Box$
\begin{lemma}
Let $s\geq 3$ be an odd integer and let $P_s(X,Y)$ be the integer polynomial from Theorem\,\ref{Thm3.3} with $n$ replaced by $s$. Then we have
\[P_s(X,0) \,=\, \prod_{\scriptsize {\begin{array}{c} u|s \\ u\geq 1 \end{array}}} {\big( \,X - u^2\,\big)}^{w(u,s/u)}\,,\]
where $w(a,b)$ is defined by the number of integers $k$ with
\[0 \,\leq \, k \,<\, b \qquad \mbox{and} \qquad \mbox{gcd\,}(a,b,k) \,=\, 1 \,.\]
\label{Lem3.2}   
\end{lemma}
{\em Proof.\/} \, The statement follows from the identity given in \cite[Lemma 4]{luca}. \hfill $\Box$
\begin{lemma}
Let $\tau \in \mathbb{H}$ be as in Theorem\,\ref{Thm3.3.3}. Then the numbers $\theta_3/\theta_4$ and $\theta_2/\theta_3$ are transcendental.
\label{Lem3.3}
\end{lemma}
{\em Proof.\/} \,\,
{\it Case 1.} $\tau$ is an algebraic number of degree $\geq 3$. \\
By Theorem\,\ref{Thm3.1}, the number 
$$
j(\tau)=256\frac{(\lambda^2-\lambda+1)^3}{\lambda^2(\lambda-1)^2}
$$
is transcendental. This implies that $\lambda=\theta^4_2/\theta^4_3$ is transcendental, and so is $\theta_2/\theta_3$. Now by using the identity
$\theta_2^4+\theta^4_4=\theta^4_3$, we conclude that the number $\theta_3/\theta_4$ is transcendental. \\
\noindent{\it Case 2.} $\tau\in\mathbb{H}$ such that $q=e^{i\pi\tau}$ is algebraic over $\mathbb{Q}$ \\
Since $q=e^{i\pi \tau}$ is an algebraic number, $\theta_3$ and $\theta_4$ are algebraically independent as well as $\theta_2$ and $\theta_3$ 
(cf.\,Theorem\,\ref{Thm3.2}). Therefore, the numbers $\theta_3/\theta_4$ and $\theta_2/\theta_3$ are transcendental.  
By Cases\,1 and 2, we complete the proof of the lemma. \hfill $\Box$
\begin{lemma} 
Let $m\geq 3$ be an integer which is either odd or  an even number of the form $2^a s$, where $a\geq 1$, and $s\geq 3$ is an odd integer.  Then, 
for any $\tau\in\mathbb{H}$ satisfying the conditions in Theorem\,\ref{Thm3.3.3}, the number 
$
\frac{\theta_3(m\tau)}{\theta_3(\tau)}
$
is transcendental. 
\label{Lem3.4}
\end{lemma}
{\em Proof.\/} \,\,We assume that $\theta_3(m\tau)/\theta_3(\tau)$  is algebraic.  By Theorem\,\ref{Thm3.3} and Lemma\,\ref{Lem3.3.1}, there exists an 
integer polynomial $T_m(X,Y)$ defined by
\[T_m(X,Y) \,:= \quad \left\{ \begin{array}{lcl}
P_m(m^2X,16Y) &,& \,\,\mbox{if} \,\,m\equiv \,1 \pmod 2 \,,\\
Q_m(X,Y) &,& \,\,\mbox{if} \,\,m\equiv \,0 \pmod 2 \,,
\end{array} \right. \]
such that 
\begin{equation*}\label{eq2.5}
\tag{3.9}
T_m\left(\frac{\theta^4_3(m\tau)}{\theta^4_3(\tau)}, \frac{\theta^4_2(\tau)}{\theta^4_3(\tau)}\right)=0. 
\end{equation*}
Now we consider the polynomial $R_m(Y)= T_m\left(\frac{\theta^4_3(m\tau)}{\theta^4_3(\tau)}, Y\right) \in \overline{\Q}[Y]$ having algebraic coefficients by our
assumption. The polynomial $R_m(Y)$ does not vanish identically: for odd integers $m$ this follows from \cite[Lemma 2.1]{elsner}, for even $m$ this is a 
consequence of Lemma\,\ref{Lem3.3.1}, cf.\eqref{eq2.1}. Hence, by \eqref{eq2.5},  we obtain 
$$
R_m\left(\frac{\theta^4_2(\tau)}{\theta^4_3(\tau)}\right)=T_m\left(\frac{\theta^4_3(n\tau)}{\theta^4_3(\tau)}, \frac{\theta^4_2(\tau)}{\theta^4_3(\tau)}\right)=0.
$$
This implies that $\theta_2/\theta_3$  is algebraic, which is a contradiction to Lemma\,\ref{Lem3.3}. Therefore, we conclude that the number
$\theta_3(m\tau)/\theta_3(\tau)$ is
transcendental.  \hfill $\Box$


\section{The proofs of our results}
\label{Sec4}
{\em Proof of Theorem\,\ref{maintheorem}.}
{\em (i)}\,\,Suppose that the functions $\theta_3(a_1 \tau),\ldots, \theta_3(a_m\tau)$  are linearly dependent over $\mathbb{C(\tau)}$. Then there
exist $c_1(\tau), \ldots, c_m(\tau)\in\mathbb{C[\tau]}$, not all zero and with minimal degree, such that
\begin{equation*}\label{eq4.1}
\tag{4.1}
c_1(\tau) \theta_3(a_1 \tau)+\cdots+c_m(\tau)\theta_3(a_m\tau)=0 \quad \mbox{~~for all ~~} \tau\in\mathbb{H}.
\end{equation*}
Let $M$ be the  common denominator of the rational numbers $a_1,\ldots,a_m$. Then $M a_j\in\mathbb{Z}$ for every $j=1,\ldots,m$, and we  notice 
that  
$$
\theta_3(a_j (\tau+2M))=1+2\sum_{n=1}^\infty e^{i\pi a_j \tau n^2} e^{2 M a_j \pi i n^2}=1+2\sum_{n=1}^\infty e^{i\pi a_j \tau n^2}=\theta_3(a_j \tau)
$$
 for all $j=1,2,\ldots, m$. Hence, the functions $\theta_3(a_1 \tau), \theta_3(a_2\tau), \ldots, \theta_3(a_m\tau)$ are periodic.  
\bigskip

Replacing $\tau$ by $\tau+2M$ and using the periodicity, we have
\begin{equation*}\label{eq4.2}
\tag{4.2}
c_1(\tau+2M) \theta_3(a_1 \tau)+\cdots+c_m(\tau+2M)\theta_3(a_m\tau)=0 \quad \mbox{~~for all ~~} \tau\in\mathbb{H}.
\end{equation*}
Thus, from \eqref{eq4.1} and \eqref{eq4.2}, we obtain
$$
(c_1(\tau)-c_1(\tau+2M))\theta_3(a_1\tau)+\cdots+(c_1(\tau)-c_m(\tau+2M))\theta_3(a_m\tau)=0 \quad \mbox{~~for all ~~} \tau\in\mathbb{H}.
$$
Note that the degree of the polynomial $c_j(\tau+2M)-c_j(\tau)$ is strictly less than the degree of the polynomial $c_j(\tau)$. Therefore, by the minimality of the
polynomials $c_1(\tau), \ldots, c_m(\tau)$, we get $c_j(\tau+2M)=c_j(\tau)$ for all $j=1,2,\ldots,m$, which in turns implies that $c_1(\tau), \ldots, c_m(\tau)$ are
constant polynomials. Hence, in order to prove that these functions are linearly independent over $\mathbb{C(\tau)}$, it suffices to prove the linear independence 
over $\mathbb{C}$.
\bigskip

Therefore we can consider the identity
$$
c_1 \theta_3(a_1 \tau)+\cdots+c_m\theta_3(a_m\tau)=0, \quad \mbox{~~for all ~~} \tau\in\mathbb{H}  \mbox{~~~and fixed~~~} c_j\in\mathbb{C}.
$$
This can be rewritten as  
\begin{equation*}\label{eq4.3}
\tag{4.3}
c_1\left(1+2\sum_{n=1}^\infty  e^{i\pi \tau a_1n^2}\right)+\cdots+c_m\left(1+2\sum_{n=1}^\infty  e^{i\pi \tau a_m n^2}\right)=0 \quad \mbox{~~for all ~~}
\tau\in\mathbb{H}.
\end{equation*}
Putting $\tau=iX$ and letting $X\rightarrow\infty$ in the above equality, we have
$$
(c_1+\cdots+c_m)+2\lim_{X\rightarrow\infty}\left(c_1\sum_{n=1}^\infty  e^{-\pi X a_1n^2}+\cdots+c_m\sum_{n=1}^\infty  e^{-\pi X a_m n^2}\right)=0.
$$
Since $\lim_{X\rightarrow\infty}\left(\sum_{n=1}^\infty  e^{-\pi X a_j n^2}\right)=0$ for all $j=1,2,\ldots,m$, we have
$$
c_1+c_2+\cdots+c_m=0.  
$$
Therefore \eqref{eq4.3} becomes  
\begin{equation*}\label{eq4.4}
\tag{4.4}
c_1\sum_{n=1}^\infty  e^{-\pi Xa_1n^2}+\cdots+c_m\sum_{n=1}^\infty  e^{-\pi X a_m n^2}=0 \quad \mbox{~~for all ~~} X>0.
\end{equation*}
Without loss of generality we can  assume that $a_1<a_2<\cdots<a_m$. Multiplying the above equality by $e^{a_1\pi X}$, we get   
\begin{equation*}\label{eq4.5}
\tag{4.5}
-c_1=c_1\sum_{n=2}^\infty  e^{-\pi Xa_1n^2+\pi X a_1 }+\Big( \,c_2\sum_{n=1}^\infty  e^{-\pi Xa_2n^2+\pi X a_1 }+\cdots
+c_m\sum_{n=1}^\infty  e^{-\pi X a_m n^2+\pi X a_1}\,\Big)
\end{equation*}
Since $-\pi a_1n^2+\pi a_1<0$ for $n\geq 2$ and $-\pi a_j n^2+\pi a_1<0$ for all $j=2,3,\ldots,m$, we see that the right-hand side of \eqref{eq4.5} tends to zero as
$X\rightarrow\infty$. Therefore, we conclude that $c_1=0$, and \eqref{eq4.4} becomes 
$$
c_2\sum_{n=1}^\infty  e^{-\pi Xa_2n^2}+\cdots+c_m\sum_{n=1}^\infty  e^{-\pi X a_m n^2}=0 \quad \mbox{~~for all ~~} X\in\mathbb{N}.
$$
Now we multiply the above equality by $e^{a_2\pi X}$ and proceed by the same process in order to get $c_2=0$. Hence, by continuing this process, we get
$c_1=c_2=\cdots=c_m=0$, which gives a contradiction. \\ 
{\em (ii)\/}\,\,As we have seen in the above proof of the first statement in Theorem 2.1, the assumption that the numbers $a_1,\ldots,a_m$ are integers, is
only used to reduce the arguments on the linear independence of $m$-theta functions $\theta_3(a_1\tau),\ldots,\theta_3(a_m\tau)$ over the field $\mathbb{C}$. Therefore, 
by using the similar approach as above, one can prove the second statement in Theorem 2.1. This proves the theorem. \hfill $\Box$ \\

\bigskip

{\em Proof of Theorem\,\ref{Thm3.3.3}.} 
Throughout this proof, we denote by $\deg_X~T(X,Y)$ and $\deg_Y~T(X,Y)$ the degree of a polynomial $T(X,Y)$ with respect to $X$ and $Y$, respectively; and 
by $\deg T(X,Y)$ we denote the total degree of the polynomial $T(X,Y)$. \vspace*{5pt} \\
It is suffices to prove  that the three numbers
$
1, \frac{\theta_3(m\tau)}{\theta_3(\tau)},  \frac{\theta_3(n\tau)}{\theta_3(\tau)}
$
are  linearly independent over $\overline{\mathbb{Q}}$. Suppose that these numbers are  linearly dependent over $\overline{\mathbb{Q}}$. Then, there exist algebraic 
integers $\alpha_0 , \alpha_1, \alpha_2$ not all zero such that
\begin{equation*}\label{eq4.6}
\tag{4.6}
\alpha_0 +\alpha_1\frac{\theta_3(m\tau)}{\theta_3(\tau)}+\alpha_2\frac{\theta_3(n\tau)}{\theta_3(\tau)}=0.
\end{equation*}
It is clear that neither $\alpha_1$ nor $\alpha_2$ vanishes, sine otherwise (when $\alpha_1=0$, $\alpha_2\not= 0$ or $\alpha_1\not= 0$, $\alpha_2= 0$) there is 
a contradiction to Lemma\,\ref{Lem3.4}, as both the numbers 
$  
\frac{\theta_3(m\tau)}{\theta_3(\tau)}\quad\mbox{and}\quad \frac{\theta_3(n\tau)}{\theta_3(\tau)}
$
are transcendental. This implies that both, $\alpha_1$ and $\alpha_2$, are non-zero. Then, when $\alpha_0=0$, we get a contradiction to \cite[Theorem\,1.1]{elsner2}.
Thus,  it implies that  $\alpha_0\alpha_1\alpha_2 \not= 0$. Then from \eqref{eq4.6} and Theorem\,\ref{Thm3.3}, we have
\begin{equation*}\label{eq4.7}
\tag{4.7}
P_m\left(m^2\left(-\frac{\alpha_0 }{\alpha_1}-\frac{\alpha_2}{\alpha_1}\frac{\theta_3(n\tau)}{\theta_3(\tau)}\right)^4, 16\frac{\theta^4_2(\tau)}
{\theta^4_3(\tau)}\right)=0.
\end{equation*}
By the explicit form of the polynomials $P_m(X, Y)$ and $P_n(X, Y)$, we see that the polynomials 
$$
H_m(X)=P_m\left(m^2{\left(-\frac{\alpha_0 }{\alpha_1}-\frac{\alpha_2}{\alpha_1}X\right)}^4,  16\frac{\theta^4_2(\tau)}{\theta^4_3(\tau)}\right)\quad\mbox{and}\quad
S_n(X)=P_n\left(n^2X^4, 16\frac{\theta^4_2(\tau)}{\theta^4_3(\tau)}\right)
$$
are non-zero. The polynomials 
$$
P_m\left(m^2\left(-\frac{\alpha_0 }{\alpha_1}-\frac{\alpha_2}{\alpha_1}\frac{\theta_3(n\tau)}{\theta_3(\tau)}\right)^4, 16Y\right)\quad\mbox{and}\quad 
P_n\left(n^2\frac{\theta^4_3(n\tau)}{\theta^4_3(\tau)}, 16Y\right)
$$ 
have the same common root $Y_0:=\theta^4_2(\tau)/\theta^4_3(\tau)$. Hence, the resultant 
\begin{equation*}\label{eq4.8}
\tag{4.8}
R(X):= \mbox{Res}_Y\left(P_m\left(m^2{\left(-\frac{\alpha_0 }{\alpha_1}-\frac{\alpha_2}{\alpha_1}X\right)}^4, 16Y\right), 
P_n(n^2 X^4, 16Y)\right)
\end{equation*}
is given by the determinant of a square matrix where the dimensions and elements of the corresponding Sylvester matrix depend on the degrees and
on the coefficients, respectively, of the polynomials
\[P_m\left(m^2\left(-\frac{\alpha_0}{\alpha_1}-\frac{\alpha_2}{\alpha_1}\right)X, 16Y\right) \]
and $P_n(n^2 X, 16Y)$. By Lemma \,\ref{Lem3.4} we know that $\theta_3(n\tau)/\theta_3(\tau)$ is transcendental. Then, from $S_{n,d_n}\not\equiv 0$ in (\ref{EQN2}), 
we have
\[S_{m,d_m}\Big( \,m^2{\Big( -\frac{\alpha_0 }{\alpha_1} - \frac{\alpha_2}{\alpha_1}\cdot \frac{\theta_3(n\tau)}{\theta_3(\tau)} \Big)}^4\,\Big) \,\not= \, 0  
\qquad \mbox{and} \qquad
S_{n,d_n}\Big( n^2\frac{\theta_3^4(n\tau)}{\theta_3^4(\tau)} \Big) \,\not= \, 0 \,.\] 
Hence, there is some real number $\delta >0$ depending on $n,m,\alpha_0,\alpha_1,\alpha_2$, and $\tau$ such that
\[S_{m,d_m}\Big( \,m^2{\Big( -\frac{\alpha_0 }{\alpha_1} - \frac{\alpha_2}{\alpha_1}X \Big)}^4\,\Big) \,\not= \, 0  \qquad \mbox{and} \qquad
S_{n,d_n}\Big( n^2X^4 \Big) \,\not= \, 0 \]
hold for
\[\Big| X -  \frac{\theta_3(n\tau)}{\theta_3(\tau)}\Big| \,<\, \delta \,.\]
Then, for all $X$ inside this circle, we have
\[\deg_Y P_m\left(m^2{\left(-\frac{\alpha_0 }{\alpha_1}-\frac{\alpha_2}{\alpha_1}X\right)}^4, 16Y\right) \,=\, 
\deg_Y P_m\left(m^2{\left(-\frac{\alpha_0 }{\alpha_1}-\frac{\alpha_2}{\alpha_1}\cdot \frac{\theta_3(n\tau)}{\theta_3(\tau)}\right)}^4, 16Y\right) \,=\, d_m \,,\]  
and, similarly,
\[\deg_Y P_n\big( n^2X^4,16Y \big) \,=\, \deg_Y P_n\Big( n^2\frac{\theta_3^4(n\tau)}{\theta_3^4(\tau)},16Y \Big) \,=\, d_n \,.\]  
For $X$ restricted to the inside of the circle with radius $\delta$ mentioned above, $R(X)$ can be considered as a polynomial in $X$ depending on the elements of a 
Sylvester matrix with fixed dimensions $d_n+d_m$. On the scale of things $R(X)$ is some polynomial with algebraic coefficients such that
\begin{equation*}\label{eq4.9}
\tag{4.9}
R\Big( \,\frac{\theta_3(n\tau)}{\theta_3(\tau)} \,\Big) \,=\,0\,,
\end{equation*}
since $Y_0$ is a common root of the polynomials under consideration. 
First we note that the polynomial $R(X)$ is not identically zero. We assume the contrary, namely that
\begin{equation*}\label{eq4.10}
\tag{4.10}
R(X) \,\equiv \,0\,.  
\end{equation*}
Then, by  \eqref{eq4.8} and \eqref{eq4.10},
$$
\mbox{Res}_Y\left(P_m\left(m^2{\left(-\frac{\alpha_0 }{\alpha_1}-\frac{\alpha_2}{\alpha_1}X\right)}^4, 16Y\right), P_n(n^2 X^4, 16Y)\right)=R(X)\equiv 0,
$$ 
and so there exists a common factor $H(X, Y)\in\mathbb{C}[X, Y]$ with positive degree in $Y$ of the polynomials 
$$
P_m\left(m^2{\left(-\frac{\alpha_0 }{\alpha_1}-\frac{\alpha_2}{\alpha_1}X\right)}^4, 16Y\right)\quad\mbox{and}\quad P_n(n^2 X^4, 16Y).
$$ 
Let 
\begin{equation*} \label{eq4.11}
\tag{4.11}
P_m\left(m^2{\left(-\frac{\alpha_0 }{\alpha_1}-\frac{\alpha_2}{\alpha_1}X\right)}^4, 16Y\right)=H(X, Y)G(X, Y).
\end{equation*}
By substituting $Y=\lambda(\tau)$ defined in the introduction into the above equation, we have
\begin{equation*}\label{eq4.12}
\tag{4.12}
P_m\left(m^2{\left(-\frac{\alpha_0 }{\alpha_1}-\frac{\alpha_2}{\alpha_1}X\right)}^4, 16\lambda(\tau)\right)=H(X, \lambda(\tau)) G\left(X, \lambda(\tau)\right).
\end{equation*}
Using the definition of $P_m(X, Y)$ in Theorem\,\ref{Thm3.3}, we have
$$
\mbox{deg}_Y~R_k(Y) \,\leq\, k \frac{m-1}{m} \,<\, k \qquad \big( 1\leq k\leq \psi(m)\big) 
$$
and, for $k=0$,
\[\mbox{deg}_Y~R_0(Y) \,=\, 0\,,\quad \mbox{since} \,\,R_0(Y)\,=\,1 \,.\]
Thus, we obtain by the right hand side of formula (\ref{eq1.1}),
\begin{align*} \label{eq4.13}
\deg P_m(X,Y) &= \max_{0\leq k\leq \psi(m)} \Big\{ \psi(m) - k + \deg_Y R_k(Y)\,\Big\} \nonumber \\
&= \psi(m) \nonumber \\
&= \deg_X P_m(X,Y)\\
\tag{4.13}
&=\deg_X P_m(X,16 \lambda(\tau)) \,.
\end{align*}
The last two identities in (\ref{eq4.13}) are a consequence of 
\[P_m(X,Y) \,=\, X^{\psi(m)} + \mathcal{O}_Y \big( X^{\psi(m)-1} \big) \,.\] 
Replacing $X$ in (\ref{eq4.13}) by
\[m^2{\left(-\frac{\alpha_0 }{\alpha_1}-\frac{\alpha_2}{\alpha_1}X\right)}^4\,,\] 
we obtain the identity
\begin{equation*} \label{eq4.14}
\tag{4.14}
\mbox{deg}_X~P_m\left(m^2{\left(-\frac{\alpha_0 }{\alpha_1}-\frac{\alpha_2}{\alpha_1}X\right)}^4, 16\lambda(\tau)\right) 
\,=\, \mbox{deg}~P_m\left(m^2{\left(-\frac{\alpha_0 }{\alpha_1} - \frac{\alpha_2}{\alpha_1}X\right)}^4,16Y\right) \,.
\end{equation*}
Hence, by the above identities, we obtain 
\begin{align*} \label{eq4.15}
\deg_X H\big( X,\lambda(\tau) \big) + \deg_X G\big( X,\lambda(\tau) \big) &= \deg_X H\big( X,\lambda(\tau) \big)  G\big( X,\lambda(\tau) \big) \nonumber \\
&\stackrel{(\ref{eq4.12})}{=} \deg_X P_m \left(m^2{\left(-\frac{\alpha_0 }{\alpha_1}-\frac{\alpha_2}{\alpha_1}X\right)}^4, 16\lambda(\tau)\right) \nonumber \\
&\stackrel{(\ref{eq4.14})}{=} \deg P_m \left(m^2{\left(-\frac{\alpha_0 }{\alpha_1}-\frac{\alpha_2}{\alpha_1}X\right)}^4, 16Y\right) \nonumber \\
&\stackrel{(\ref{eq4.11})}{=} \deg H(X,Y)G(X,Y) \nonumber \\
\tag{4.15}
&\hspace*{7pt}= \deg H(X,Y) + \deg G(X,Y) \,.
\end{align*}
Additionally, we have the obvious inequalities
\begin{equation*} \label{eq4.16}
\tag{4.16}
\deg_X H\big( X,\lambda(\tau) \big) \,\leq \, \deg H(X,Y) \qquad \mbox{and} \qquad \deg_X G\big( X,\lambda(\tau) \big) \,\leq \, \deg G(X,Y) \,.
\end{equation*}
Thus, we obtain from (\ref{eq4.15}) and (\ref{eq4.16}) that $\deg_X H(X,\lambda(\tau))=\deg H(X,Y)$ (and, similarly,  $\deg_X G(X,\lambda(\tau))=\deg G(X,Y)$), 
and consequently
\begin{equation*}\label{eq4.17}
\tag{4.17}
\mbox{deg}_X~H(X,\lambda(\tau)) \geq \mbox{deg}_Y~H(X,Y)\geq 1.
\end{equation*}
By Lemma \ref{Lem3.3}, $\frac{\theta_2(\tau)}{\theta_3(\tau)}$ is transcendental in each of the two cases $\tau$ algebraic of degree $\geq 3$ and $\tau$ such that 
$e^{i\pi\tau}$ is algebraic, hence the number $16\lambda(\tau)$ is transcendental over ${\Q}$. So, it is also transcendental over $\overline{\Q}$. Let
\[\beta_m \,:=\, \frac{\theta_3(m\tau)}{\theta_3(\tau)} \,.\]
By Theorem\,\ref{Thm3.3}, we know that $P_m\big( m^2\beta_m^4,16\lambda(\tau) \big)=0$. Moreover, it follows from \cite[Lemma\,2]{luca} that the polynomial
$P_m\big( X,16\lambda(\tau) \big)$ is irreducible over the field ${\K}:={\overline{\Q}}\big( \lambda(\tau) \big)$. Overall, we may consider $\beta_m$ as
an algebraic number over ${\K}$, where,  by (\ref{eq1.1}),
\begin{equation}\label{4.18}
\tag{4.18}
\deg_{\K}\beta_m \,=\, 4\deg_X P_m(X,Y) \,=\, 4\psi(m) \,.
\end{equation}
Let us assume that the polynomial
\begin{equation}\label{4.19}
\tag{4.19}
P_m\Big( \,m^2{\Big( \,-\frac{\alpha_0}{\alpha_1} -\frac{\alpha_2}{\alpha_1}X \,\Big)}^4,16\lambda(\tau)\,\Big)  
\end{equation}
is reducible over ${\K}$. Then we obtain the inequality
\[\deg_{\K} \Big( \,m^2{\Big( \,-\frac{\alpha_0}{\alpha_1} -\frac{\alpha_2}{\alpha_1}\beta_m \,\Big)}^4\,\Big) \,<\, 4\psi(m) \,.\]
But this is impossible by (\ref{4.18}), and by 
\[m^2,\,\frac{\alpha_0}{\alpha_1},\,\frac{\alpha_2}{\alpha_1} \,\in\, \overline{\Q}\setminus \{ 0 \}\,,\]
since $\beta_m$ is transcendental over $\overline{\Q}$ by the identity $P_m(m^2 \beta^4_m, 16 \lambda(\tau))=0$ and by the 
transcendence of $\lambda(\tau)$, as noticed earlier.
The contradiction proves the irreducibility of the polynomial in (\ref{4.19}). \vspace*{5pt} \\
Thus, from \eqref{eq4.11} and \eqref{eq4.17},  we obtain
$$
P_m\left(m^2{\left(-\frac{\alpha_0 }{\alpha_1}-\frac{\alpha_2}{\alpha_1}X\right)}^4, 16\frac{\theta^4_2(\tau)}{\theta^4_3(\tau)}\right)=\beta_1 H(X, \lambda(\tau))
$$
for some non-zero complex number $\beta_1$.  Similarly, there exists a non-zero complex number $\beta_2$  such that
$$
P_n(n^2 X^4, 16\lambda(\tau))=\beta_2  H(X, \lambda(\tau)),
$$
and hence
$$
P_m\left(m^2{\left(-\frac{\alpha_0 }{\alpha_1}-\frac{\alpha_2}{\alpha_1}X\right)}^4, 16\lambda(\tau)\right)=cP_n(n^2 X^4, 16\lambda(\tau)), \quad c:=\beta_1/\beta_2.
$$
This polynomial identity holds for all complex numbers $\tau \in \mathbb{H}$. We know that for $\tau\rightarrow i\infty$ 
the function $\lambda(\tau)$ tends to zero. Hence, taking $\tau\rightarrow i\infty$ into the above equality, we have by Lemma\,\ref{Lem3.2}, 
$$
\prod_{d|m}\left(m^2{\left(-\frac{\alpha_0 }{\alpha_1}-\frac{\alpha_2}{\alpha_1}X\right)}^4-d^2\right)^{w(d,m/d)}=c\prod_{d|n}(n^2 X^4-d^2)^{w(d,n/d)} \,.
$$
Then, comparing the multiplicity of the zero of these polynomials at $X=-\left(\alpha_0 +\alpha_1/\sqrt{m}\right)/\alpha_2$ (and $d=1$ on the left-hand side), we obtain
$$
m=w(1,m)\leq\max_{d}~w(d,n/d)\leq n,
$$
which is a contradiction to the condition $n<m$ from the theorem. Hence, the polynomial $R(X)$ is non-zero.  Therefore, it follows from \eqref{eq4.9} 
that the number $\theta_3(n\tau)/\theta_3(\tau)$ is algebraic, which is a contradiction to the fact from Lemma\,\ref{Lem3.4} that the number
$\theta_3(n\tau)/\theta_3(\tau)$ 
is transcendental. This proves the assertion. \hfill $\Box$
\bigskip

{\em Proof of Theorem\,\ref{Thm3.3.4}.} \,Let $m=2^a s_1$ and $n=2^b s_2$ be two different integers with $a,b\geq 1$ and odd integers $s_1,s_2\geq 3$. By  
Lemma\,\ref{Lem3.3.1} there exist integer polynomials $Q_m(X, Y)$ and $Q_n(X, Y)$  such that 
$$
Q_m\left(\frac{\theta^4_3(m\tau)}{\theta^4_3(\tau)}, \frac{\theta^4_2(\tau)}{\theta^4_3(\tau)}\right)=0 \qquad \mbox{and} \qquad
Q_n\left(\frac{\theta^4_3(n\tau)}{\theta^4_3(\tau)}, \frac{\theta^4_2(\tau)}{\theta^4_3(\tau)}\right)=0.
$$
We assume that the linear equation \eqref{eq4.6} holds, where $\alpha_0,\alpha_1,\alpha_2$ are algebraic numbers satisfying the hypothesis in 
Theorem\,\ref{Thm3.3.4}. As in the proof of Theorem\,\ref{Thm3.3.3},  we have $\alpha_0\alpha_1\alpha_2 \not= 0$. By the hypotheses of the theorem, we may assume without
loss of generality that $\beta:=\alpha_2/\alpha_0$ satisfies $R_{n,0}(\beta^{-4})\not=0$. Namely, it is obvious by (\ref{eq2.4}) and Lemma\,\ref{Lem3.2} that 
$R_{n,0}(\beta^{-4})\not=0$ holds particularly for $\beta^{-1} \not\in M_{s_2}$ and for $\beta \not\in {\Q}$. \vspace*{5pt} \\
Then we obtain
$$
Q_m\left(\left(-\frac{\alpha_0}{\alpha_1}-\frac{\alpha_2}{\alpha_1}\frac{\theta_3(n\tau)}{\theta_3(\tau)}\right)^4, \frac{\theta^4_2(\tau)}{\theta^4_3(\tau)}\right)=0.
$$
By the explicit form of the polynomials $Q_m(X, Y)$ and $Q_n(X, Y)$, we see that the polynomials 
$$
Q_m\left({\left(-\frac{\alpha_0}{\alpha_1}-\frac{\alpha_2}{\alpha_1}X\right)}^4,  \frac{\theta^4_2(\tau)}{\theta^4_3(\tau)}\right) \qquad \mbox{and} \qquad 
Q_n\left(X^4, \frac{\theta^4_2(\tau)}{\theta^4_3(\tau)}\right)
$$
are non-zero. Hence, the polynomials 
$$
Q_m\left(\left(-\frac{\alpha_0}{\alpha_1}-\frac{\alpha_2}{\alpha_1}\frac{\theta_3(n\tau)}{\theta_3(\tau)}\right)^4, Y\right) \qquad \mbox{and} \qquad 
Q_n\left( \frac{\theta^4_3(n\tau)}{\theta^4_3(\tau)}, Y\right)
$$
have the same common root $Y_0=\theta^4_2(\tau)/\theta^4_3(\tau)$.  

\noindent 
Let
\[H_m(X,Y) \,:=\, Q_m\left(\left(-\frac{\alpha_0}{\alpha_1}-\frac{\alpha_2}{\alpha_1}X\right)^4, Y\right) \]
and
\[W(X) \,:=\, \mbox{Res}_Y \big( H_m(X,Y),\,Q_n(X^4,Y) \big) \,\in \, \overline{\Q}[X] \,.\]
From Lemma\,\ref{Lem3.3.1} we know that both, $\deg_Y Q_m(X,Y)$ and $\deg_Y Q_n(X,Y)$, do not depend on $X$, since the coefficients of the leading terms with respect to 
$Y$ are non-zero integers. Thus, $W(X)$ can be considered as a polynomial for all $X$. \\
In order to show that the polynomial $W(X)$ does not vanish identically, we shall prove the existence of a number $\eta$ satisfying $W(\eta)\not= 0$, or,
equivalently, that the polynomials $H_m(\eta,Y)$ and $Q_n(\eta^4,Y)$ are coprime. Let
\[\eta \,:=\, -\frac{1}{\beta} \,=\, -\frac{\alpha_0}{\alpha_2} \,.\]
On the one side, by using \eqref{eq2.3}, we obtain  
\[H_m(\eta,Y) \,=\, Q_m(0,Y) \,=\, c_m^{2^a}Y^{2^a \psi(s_1)}\,.\]
Therefore, $H_m(\eta,Y)$ is a nonvanishing polynomial in $Y$ having exclusively a multiple root at $Y=0$. \\
On the other side, by applying formulas \eqref{eq2.1} and (\ref{eq2.2}) in Lemma\,\ref{Lem3.3.1}, we have
\[Q_n(\eta^4,Y) \,=\, c_n^{2^b}Y^{2^b\psi(s_2)} + \sum_{j=0}^{2^b\psi(s_2)-1} R_{n,j}(\eta^4)Y^j \]
with $c_n\in {\Z}\setminus \{ 0\}$ and $R_{n,0}(X)\not\equiv 0$ by (\ref{eq2.4}) and Lemma\,\ref{Lem3.2}. We already know by $\eta=-1/\beta$ that $R_{n,0}(\eta^4) \not= 0$.
Consequently, we have $Q_n(\eta^4,0)=R_{n,0}(\eta^4) \not= 0$. \\
Altogether, the polynomials $H_m(\eta,Y)$ and $Q_n(\eta^4,Y)$ have no common root. More precisely, we obtain for $W(X)$,
\[W(\eta) \,=\, \mbox{Res}_Y \big( H_m(\eta,Y),Q_n(\eta^4,Y) \big) \,\not= \, 0 \,.\]
This shows that $W(X)$ does not vanish identically. By construction, we know that $W(X_0)$ vanishes for
$X_0 \,:=\, \frac{\theta_3(n\tau)}{\theta_3(\tau)} \,,$
which implies the algebraicity of $\theta_3(n\tau)/\theta_3(\tau)$, a contradiction to Lemma\,\ref{Lem3.4}. This finally shows that the linear relation \eqref{eq4.6}
cannot hold.
 \hfill $\Box$
\bigskip

\noindent {\em Proof of Theorem\,\ref{Thm3.3.5}.} \,Let $m=2^a s$ and $n$ be two integers with $a\geq 1$ and odd integers $n,s\geq 3$. By Theorem\,\ref{Thm3.3} and
Lemma\,\ref{Lem3.3.1} there exist integer polynomials $P_n(X,Y)$ and $Q_m(X,Y)$ such that
$$
Q_m\left(\frac{\theta^4_3(m\tau)}{\theta^4_3(\tau)}, \frac{\theta^4_2(\tau)}{\theta^4_3(\tau)}\right)=0 \qquad \mbox{and} \qquad
P_n\left(n^2\frac{\theta^4_3(n\tau)}{\theta^4_3(\tau)}, 16\frac{\theta^4_2(\tau)}{\theta^4_3(\tau)}\right)=0.
$$
We assume that the linear equation \eqref{eq4.6} holds. As in the proof of Theorem\,\ref{Thm3.3.3} we have $\alpha_0\alpha_1\alpha_2 \not= 0$. By the hypotheses 
of the theorem,  we may assume that $\beta:=\alpha_2/\alpha_0$ satisfies either $\deg_{\Q} (\beta^4) > \psi(n)$, or $S_{n,0}(n^2\beta^{-4})S_{n,d_n}(n^2\beta^{-4}) 
\not= 0$. Then we obtain
$$
Q_m\left(\left(-\frac{\alpha_0}{\alpha_1}-\frac{\alpha_2}{\alpha_1}\frac{\theta_3(n\tau)}{\theta_3(\tau)}\right)^4, \frac{\theta^4_2(\tau)}{\theta^4_3(\tau)}\right)=0.
$$
By the explicit form of the polynomials $Q_m(X, Y)$ and $P_n(X, Y)$ given by Theorem\,\ref{Thm3.3} and Lem- ma\,\ref{Lem3.3.1}, we see that the polynomials 
$$
Q_m\left({\left(-\frac{\alpha_0}{\alpha_1}-\frac{\alpha_2}{\alpha_1}X\right)}^4,  \frac{\theta^4_2(\tau)}{\theta^4_3(\tau)}\right) \qquad \mbox{and} \qquad 
P_n\left(n^2X^4, 16\frac{\theta^4_2(\tau)}{\theta^4_3(\tau)}\right)
$$
are non-zero. Hence, the polynomials 
$$
Q_m\left(\left(-\frac{\alpha_0}{\alpha_1}-\frac{\alpha_2}{\alpha_1}\frac{\theta_3(n\tau)}{\theta_3(\tau)}\right)^4, Y\right) \qquad \mbox{and} \qquad 
P_n\left( n^2\frac{\theta^4_3(n\tau)}{\theta^4_3(\tau)}, 16Y\right)
$$
have the same common root $Y_0=\theta^4_2(\tau)/\theta^4_3(\tau)$. Let
\[H_m(X,Y) \,:=\, Q_m\left(\left(-\frac{\alpha_0}{\alpha_1}-\frac{\alpha_2}{\alpha_1}X\right)^4, Y\right) \]
and
\[W(X) \,:=\, \mbox{Res}_Y \big( H_m(X,Y),\,P_n(n^2X^4,16Y) \big) \,.\]
From Lemma\,\ref{Lem3.3.1}, formula (\ref{eq2.1}), we know that $\deg_Y Q_m(X,Y)$ (and, consequently, $\deg_Y H_m(X,Y)$) does not depend on $X$, since 
the coefficient of the leading term with respect to $Y$ is the non-zero integer $c_m^{2^a}$. For all real numbers $X$ which are not a root of the polynomial 
$S_{n,d_n}(X)$ in (\ref{EQN1}), the leading term of $P_n(X,Y)$ with respect to $Y$ does not vanish. Consequently, $W(X)$ is given by the same polynomial for all
these $X$, since the degrees $\deg_Y H_m(X,Y)$ and $\deg_Y P_n(n^2X^4,16Y)$ do not depend on all $X$ satisfying $S_{n,d_n}(n^2X^4)\not= 0$. Note that 
$S_{n,d_n}(X)\not\equiv 0$ by (\ref{EQN2}). \\
In order to show that $W(X)$ does not vanish identically for $X$ with $S_{n,d_n}(n^2X^4)\not= 0$, we shall prove the existence of a number $\eta$ satisfying
$W(\eta)\not= 0$, or,
equivalently, that the polynomials $H_m(\eta,Y)$ and $P_n(n^2\eta^4,16Y)$ are coprime. Let
\[\eta \,:=\, - \frac{1}{\beta} \,=\, -\frac{\alpha_0}{\alpha_2} \,.\]
On the one side, by using \eqref{eq2.3}, we obtain  
\[H_m(\eta,Y) \,=\, Q_m(0,Y) \,=\, c_m^{2^a}Y^{2^a \psi(s)}\,.\]
Therefore, $H_m(\eta,Y)$ is a nonvanishing polynomial in $Y$ having only a multiple root at $Y=0$. \\
On the other side, by the hypothesis on $\beta$ in Theorem\,\ref{Thm3.3.5} and by 
\[\deg_X S_{n,d_n}(X) \,\leq \, \deg_X P_n(X,Y) \,=\, \psi(n) \]
(cf. (\ref{EQN1}) in Theorem\,\ref{Thm3.3}), we know that
\[S_{n,d_n}\big( n^2\eta^4 \big) \,=\, S_{n,d_n}\Big( \frac{n^2}{\beta^4}\,\Big) \,=\, S_{n,d_n}\Big( \frac{n^2\alpha_0^4}{\alpha_2^4}\,\Big) 
\,\not=\, 0 \,.\]
This shows that the degree with respect to $Y$ of the polynomial on the right-hand side of (\ref{EQN1}) does not change for the particular choice of $X=n^2\eta^4$.
Moreover, it follows from (\ref{EQN3}) that $S_{n,0}\not\equiv 0$, and therefore the inequality
\[\deg_X S_{n,0}(X) \,\leq \, \deg_X P_n(X,Y) \,=\, \psi(n) \]
and the conditions on $\beta$ imply that 
\[S_{n,0}\big( n^2\eta^4 \big) \,=\, S_{n,d_n}\Big( \frac{n^2}{\beta^4}\,\Big) \,=\, S_{n,0}\Big( \frac{n^2\alpha_0^4}{\alpha_2^4}\,\Big) \,\not=\, 0 \,.\]
Thus, again the application of (\ref{EQN1}) gives 
$P_n\big( n^2\eta^4,0 \big) \,\not= \, 0\,.$
Altogether, the polynomials $H_m(\eta,Y)$ and $P_n(n^2\eta^4,16Y)$ have no common root. More precisely, we obtain that
\[W(\eta) \,=\, \mbox{Res}_Y \big( H_m(\eta,Y),P_n(n^2\eta^4,16Y) \big) \,\not= \, 0 \,.\]
This shows that $W(X)$ does not vanish identically for all $X$ satisfying $S_{n,d_n}(n^2X^4)\not= 0$. By construction, we know that $W(X_0)$ vanishes for
$X_0 \,:=\, \frac{\theta_3(n\tau)}{\theta_3(\tau)},$
and since $\theta_3(n\tau)/\theta_3(\tau)$ is transcendental by Lemma\,\ref{Lem3.4}, we have by (\ref{EQN2}) that
\[S_{n,d_n}\Big( \,n^2\frac{\theta_3^4(n\tau)}{\theta_3^4(\tau)}\,\Big) \,\not= \, 0\,.\]
Thus, $X=X_0$ is a zero of the function $W(X)$, which restricted to all values $X$ satisfying $S_{n,d_n}(n^2X^4)\not= 0$ results in the same nonvanishing polynomial
$W(X)$. This implies the algebraicity of $\theta_3(n\tau)/\theta_3(\tau)$, a contradiction to Lemma\,\ref{Lem3.4}. This finally shows that the linear relation
\eqref{eq4.6} cannot hold. \hfill $\Box$
\smallskip

\noindent{\em Proof of Proposition\,\ref{Prop1}.\/}\quad
Replacing $\tau$ by $2^m\tau$, it suffices to prove the assertion for the three numbers
$$
\theta_3(\tau),\quad \theta_3(2\tau),\quad \theta_3(4\tau).
$$
We have the following identies:
\begin{align*}\label{eq3.1a}
2\theta^2_3(2\tau)&=\theta^2_3+\theta^2_4,\\
2\theta_3(4\tau)&=\theta_3+\theta_4.\ 
\tag{4.20}  
\end{align*}
Suppose there exist algebraic numbers $\alpha, \beta$, $\gamma$ not all zero such that
\begin{equation*}\label{eq3.2a}
\tag{4.21}
2\alpha\theta_3(\tau)+2\beta\theta_3(2\tau)+2\gamma\theta_3(4\tau)=0.
\end{equation*} 
Substituting \eqref{eq3.1a} into \eqref{eq3.2a}, we get
\begin{equation*}
2\alpha\theta_3+2\beta\sqrt{\frac{\theta^2_3+\theta^2_4}{2}}+\gamma(\theta_3+\theta_4)=0.
\end{equation*}
By rearranging this formula, we get
\begin{equation*}\label{eq3.3a}
\tag{4.22}
\left((2\alpha+\gamma)^2-2\beta^2\right)^2\theta^2_3+(\gamma^2-2\beta^2)\theta^2_4+
2\gamma(2\alpha+\gamma)\theta_3 \theta_4=0.
\end{equation*}
Dividing \eqref{eq3.3a}, by $\theta^2_3$, we obtain
\[\left((2\alpha+\gamma)^2-2\beta^2\right)^2+(\gamma^2-2\beta^2)\left(\frac{\theta_4}{\theta_3}\right)^2+
2\gamma(2\alpha+\gamma) \frac{\theta_4}{\theta_3}=0\,.\]
Hence, by Lemma\,\ref{Lem3.3}, we have 
$$
(2\alpha+\gamma)^2-2\beta^2=0, \quad \gamma^2-2\beta=0, \quad 2\gamma(2\alpha+\gamma)=0.
$$
Thus, we conclude that   
$\alpha=\beta=\gamma=0$. This proves Proposition\,\ref{Prop1}. \hfill $\Box$
\section{Concluding remarks.}  In the case when $\tau\in\mathbb{H}$  such that $e^{i\pi\tau}$ is algebraic, the number $\theta_3(\tau)$ is transcendental due to the algebraic
independence of the values $\theta_3(m\tau)$  and $\theta_3(n\tau)$ for distinct positive integers $m, n$. By our Theorem 2, we know that at least two of the numbers among 
$\theta_3(\tau), \theta_3(m\tau)$ and $\theta_3(n\tau)$ are transcendental for any $\tau\in \mathbb{H}$ such that either $\tau$ is algebraic of degree $\geq 3$ or 
$e^{i\pi \tau}$ is algebraic.  In this context, it is interesting to consider the following problem:

\bigskip
 
\noindent{\bf Problem 1.~} Let $\tau$ and $m, n$ be  as in Theorem 2.  Then $1,\theta_3(\tau), \theta_3(m\tau)$ and $\theta_3(n\tau)$ are $\overline{\mathbb{Q}}$- linearly
independent.
\smallskip

As a consequence of this problem, one can conclude the transcendence of $\theta_3(\tau)$ for algebraic $\tau$ of degree $\geq 3$, which is not known in  this case.
\bigskip

In this paper we considered linear forms in three  values of theta constant $\theta_3$. It is natural to consider the following more general problem: 
\bigskip

\noindent{\bf Problem 2.~}  Let $a_1, a_2,\ldots, a_m$ be distinct positive integers.  Let $\alpha_0,\ldots,\alpha_m$ be non-zero algebraic numbers. Under what values of $\tau\in\mathbb{H}$, the linear form
$$
L:=\alpha_1 \theta_3(a_1\tau)+\cdots+\alpha_m\theta_3(a_m \tau)
$$
does not vanish.
In the case  $\alpha_i$'s are rational and $\tau=\frac{i\log b}{2\pi}$, where  $b\geq 2$ is an integer, $L\neq 0$ by the result from  \cite{veekesh} and \cite{tachiya}.

\bigskip

\section{Acknowledgements.}
\label{Sec5}
We would like to express our deep gratitude to Professor Y.\,Tachiya for his useful comments. 
We are also indebted to the unknown reviewer for his helpful comments to improve the
manuscript.



\begin{thebibliography}{9}
\bibitem{Berndt}
Berndt, B.C., {\it Ramanujan's Notebook}, part\,V, Springer-Verlag, New York, 1998.
\bibitem{Bertrand}
Bertrand, D., {\it Theta functions and transcendence}, Ramanujan J. {\bf 1}, no.\,4 (1997), 339-350.
\bibitem{elsner}
Elsner, C., Tachiya, Y., {\em Algebraic results for certain values of the Jacobi Theta - Constant $\theta_3(\tau)$\/},
Mathematica Scandinavica {\bf 123}, no.\,2 (2018), 249-272.
\bibitem{luca}
Elsner, C., Luca, F. and Tachiya, Y., {\it Algebraic results for the values $\theta_3(m\tau)$ and $\theta_3(n\tau)$ of the Jacobi-theta constant}, 
Mosc. J. Comb. Number Theory, vol 8, (2019), 71-79.
\bibitem{elsner2}
Elsner, C., Kaneko, M. and Tachiya, Y., {\em Algebraic independence results for the values of the theta-constants and some identities\/}, J.\,Ramanujan 
Math.\,Soc. ({\bf 35}),\,1 (2020), 71-80.
\bibitem{veekesh}
Karmakar, D., Kumar, V. and Thangadurai, R., {\it Linear independence of special values of Jacobi theta-constants\/}, submitted.
\bibitem{nest2}
Nesterenko, Yu. V., {\it On some identities for theta-constants\/},  
Diophantine analysis and related fields 2006, 151--160, {\em Sem. Math. Sci.} {\bf 35}, Keio Univ., Yokohama, 2006. 
\bibitem{tachiya}
Shintaro, M., Tachiya, Y., {\it Linear independence of certain numbers in the base-b number system\/}, Arch. Math. (Basel) 122 (2024), no. 1, 31-40.
\bibitem{Schneider}
Schneider, Th., {\it Arithmetische Untersuchungen elliptischer Integrale}, Math. Annalen {\bf 113}  (1937), 1-13.
\end{thebibliography}
\end{document}